\documentclass[reqno,10pt,a4paper]{amsart}

\usepackage{amsmath}
\usepackage{amsthm}
\usepackage{amssymb}
\usepackage[mathscr]{eucal}
\usepackage{mathtools}
\usepackage{microtype}
\usepackage[T1]{fontenc}
\usepackage{lmodern} 
\usepackage[colorlinks=true,linkcolor=black,citecolor=black]{hyperref}
\usepackage[colorinlistoftodos,textsize=tiny]{todonotes}

\newtheorem{thm}{Theorem}[section]
\newtheorem{prop}[thm]{Proposition}
\newtheorem{lemma}[thm]{Lemma}
\newtheorem{cor}[thm]{Corollary}
\theoremstyle{definition}

\newtheorem{remark}[thm]{Remark}

\numberwithin{equation}{section}

\DeclareMathOperator{\kap}{Cap}
\DeclareMathOperator{\Ker}{Ker}
\DeclareMathOperator{\ind}{ind}
\DeclareMathOperator{\Real}{Re}

\let\epsilon=\varepsilon
\let\phi=\varphi

\def\R{\mathbb{R}}
\def\Z{\mathbb{Z}}
\def\N{\mathbb{N}}
\def\Nzero{\N\setminus\{0\}}

\newcommand{\Ab}{\mathbf{A}}

\def\l{\mathfrak{l}}

\def\RDK{T^{R\to D}_K}
\def\RDO{T^{R\to D}_\Omega}
\def\DRK{T^{D\to R}_K}
\def\DRO{T^{D\to R}_\Omega}
\def\DRq{T^{D\to R}_q}

\title[Clusters for the magnetic Laplacian
with Robin condition]%
{Clusters of eigenvalues for the magnetic Laplacian
with Robin condition}
\author{Magnus Goffeng}
\author{Ayman Kachmar}
\author{Mikael Persson Sundqvist}
\email[M. Goffeng]{goffeng@math.uni-hannover.de}
\email[A. Kachmar]{ayman.kashmar@liu.edu.lb}
\email[M. Persson Sundqvist]{mickep@maths.lth.se}
\begin{document}

\begin{abstract}
We study the Schr\"odinger operator with a constant magnetic field in the 
exterior of a compact domain in euclidean space. Functions in the domain of 
the operator are subject to  a boundary condition of the third type (a magnetic
Robin condition). In addition to the Landau levels, we obtain that the spectrum 
of this operator consists of clusters of eigenvalues around the Landau levels 
and that they do accumulate to the Landau levels from below. We give a precise 
asymptotic formula for the rate of accumulation of eigenvalues in these 
clusters, which is independent of the boundary condition.
\end{abstract}

\maketitle

\makeatletter
\providecommand\@dotsep{5}
\makeatother

\section{Introduction}
Magnetic Schr\"odinger operators in domains with boundaries appear
in several areas of physics. E.g. the Ginzburg--Landau
theory of superconductors, the theory of Bose--Einstein condensates,
and of course the study of edge states in Quantum mechanics. We
refer the reader to \cite{AfHe, FH, hosm} for details and additional
references on the subject. From the point of view of spectral
theory, the presence of boundaries has an effect similar to that of
perturbing the magnetic Schr\"odinger operator by an electric
potential. In both cases, the essential spectrum
consists of the Landau levels and the discrete spectrum form
clusters of eigenvalues around the Landau levels. Several papers are
devoted to the study of different aspects of these clusters of
eigenvalues in domains with or without boundaries. For results in the 
semi-classical context, see~\cite{FK, Fr, HM, Kjmp, Krmp}. In case of
domains with boundaries, see~\cite{P, puro}.

Let us consider a \emph{compact} domain
$K\subset\R^{2d}$ with Lipschitz boundary. Let us
denote by $K^\circ$ the interior of $K$, $\Omega=\R^{2d}\setminus K$ and 
$\partial\Omega$ the common 
boundary of $\Omega$ and $K$. Given a real valued function 
$\tau\in L^\infty(\partial\Omega,\R)$ and a positive constant $b$ (the 
intensity of the magnetic field), we define the Schr\"odinger operator
$L_{\Omega,b}^\tau$ with domain $D(L_{\Omega,b}^\tau)$ as
follows,
\begin{align}
D\bigl(L_{\Omega,b}^\tau\bigr)&=\bigl\{u\in L^2(\Omega)~:~(\nabla-ib\Ab_0)^ju\in
L^2(\Omega),~j=1,2;\nonumber\\ 
&\qquad\qquad\qquad\qquad\nu_\Omega\cdot(\nabla-ib\Ab_0)u+\tau
u=0\quad\text{on }\partial\Omega\bigr\}\,,\label{D-L}\\
L_{\Omega,b}^\tau u
&=-(\nabla-ib\Ab_0)^2u\quad\forall~u\in
D\bigl(L_{\Omega,b}^\tau\bigr)\,.\label{Op-L}
\end{align}
Here, $\Ab_0$ is the magnetic potential in the symmetric gauge defined by
\begin{equation}\label{mp-A0}
\Ab_0(x_1,x_2,\ldots, x_{2d})=\tfrac12(-x_2,x_1,\ldots, -x_{2d},x_{2d-1})\,,
\end{equation}
and $\nu_\Omega$ is the unit {\it outward} normal vector of the
boundary $\partial\Omega$. We also introduce the boundary Neumann and Robin 
differential notations
\begin{equation}
\label{eq:dNdR}
\partial_N = \nu_\Omega\cdot(\nabla-i b\Ab_0),\quad\text{and}\quad
\partial_R = \partial_N+\tau=\nu_\Omega\cdot(\nabla-i b\Ab_0)+\tau.
\end{equation}
The operator $L_{\Omega,b}^\tau$ is actually the Friedrich's
self-adjoint extension in $L^2(\Omega)$ associated with the
semi-bounded quadratic form
\begin{equation}\label{QF-q}
\l_{\Omega,b}^\tau(u)=\int_\Omega|(\nabla-ib\Ab_0)u|^2\,dx
+\int_{\partial\Omega}\tau|u|^2\,d S\,,
\end{equation}
defined for all functions $u$ in the form domain
\begin{equation}\label{D-QF-q}
D\bigl(\l_{\Omega,b}^\tau\bigr)=H^1_{\Ab_0}(\Omega)=
\bigl\{u\in L^2(\Omega)~:~(\nabla-ib\Ab_0)u\in
L^2(\Omega)\bigr\}\subseteq H^1_{\text{loc}}(\Omega)\,.
\end{equation}

We introduce the \emph{(multidimensional) Landau levels} 
$\Lambda_q$, $q\in\N$, as
\[
\Lambda_q := \bigl(2(q-1)+d\bigr)b,\quad q\in\Nzero\,,
\quad \Lambda_0:=-\infty.
\]
The name is motivated by the fact that these numbers (for $q\in\Nzero$) are the
eigenvalues of the Landau Hamiltonian in $\R^{2d}$, see 
Section~\ref{sec:landau}.

We are now able to state the first main result, concerning the essential 
spectrum of $L_{\Omega,b}^\tau$ together with the non-accumulation of 
eigenvalues to the Landau levels from above.

\begin{thm}
\label{thm-GKPone}
Let $\Omega\subset \R^{2d}$ be a compactly complemented Lipschitz domain 
and $\tau\in L^\infty(\partial\Omega,\R)$. The essential spectrum 
of the operator $L_{\Omega,b}^\tau$
consists of the Landau levels,
\begin{equation}\label{eq-thm-1}
\sigma_{\text{ess}}\bigl(L_{\Omega,b}^\tau\bigr)
=\bigl\{\Lambda_q~:~q\in\Nzero\bigr\}\,.
\end{equation}
Moreover, for all $\epsilon\in(0,b)$ and $q\in\Nzero$, the spectrum
of $L_{\Omega,b}^\tau$ in the open interval
$\bigl(\Lambda_q,\Lambda_q+\varepsilon\bigr)$ is finite.
\end{thm}

Next, we restrict our attention to the case that $\partial\Omega$ is $C^\infty$ 
and $\tau\in C^\infty(\partial\Omega,\R)$. For $\epsilon>0$, we let 
$N\bigl(\Lambda_{q-1},\Lambda_q -\epsilon,L_{\Omega,b}^\tau\bigr)$ 
denote the number of eigenvalues 
of $L_{\Omega,b}^\tau$ in the open interval 
$\bigl(\Lambda_{q-1},\Lambda_q -\epsilon\bigr)$, counting multiplicity. 
We also denote by 
$\kap(K)$ the logarithmic capacity of the domain $K=\R^2\setminus\Omega$.

\begin{thm}
\label{thm-GKPtwo}
Let $\Omega\subset \R^{2d}$ be a compactly complemented domain with smooth 
boundary and $\tau\in C^\infty(\partial\Omega,\R)$.
\begin{itemize}\item[(A)] Assume $d=1$. For all $q\in\Nzero$, denoting by
$\bigl\{\ell_j^{(q)}\bigr\}_{j\in\mathbb N}$ the nondecreasing sequence of
eigenvalues of $L_{\Omega,b}^\tau$ (counting multiplicities) in the interval
$(\Lambda_{q-1},\Lambda_q)$, the following holds:
\begin{equation}
\label{eq-thm-2}
\lim_{j\to{+\infty}}
\Bigl(j!\bigl(\Lambda_q-\ell_j^{(q)}\bigr)\Bigr)^{1/j}
=\frac{b}2\big{(}\kap(K)\big{)}^2\,.
\end{equation} 
\item[(B)] Assume $d\geq 1$. The counting function 
$N\bigl(\Lambda_{q-1},\Lambda_q -\epsilon,L_{\Omega,b}^\tau\bigr)$ has 
the asymptotics:
\begin{equation}
\label{eq-thm-higherd}
N\bigl(\Lambda_{q-1},\Lambda_q -\epsilon,L_{\Omega,b}^\tau\bigr)
\sim \begin{pmatrix}q+d-1\\d-1\end{pmatrix} \frac{1}{d!}
\left( \frac{|\log\epsilon|}{\log|\log \epsilon|}\right)^d 
\quad\text{as}\quad \epsilon\to 0^+.
\end{equation}
\end{itemize}
\end{thm}

\begin{remark}
Theorems~\ref{thm-GKPone} and~\ref{thm-GKPtwo} were obtained for the 
Neumann case ($\tau\equiv0$) by 
the third author in \cite{P}, and our proofs will follow the same idea as
in~\cite{P}. 
However, in~\cite{P}, the full details concerning the reduction 
to Toeplitz type operators were not written out explicitly. In this paper, 
we aim not only to generalize the Neumann result, but also to make the 
proof of Theorem~\ref{thm-GKPtwo} more transparent.
\end{remark}

\begin{remark}
Our proof of Theorem \ref{thm-GKPtwo} is carried out for $\tau$ being a 
self-adjoint pseudo-differential operator of  
order $0$ on $\partial \Omega$. They can be generalized to self-adjoint 
pseudo-differential operators of order $t<1$.
We also note that the pseudo-differential nature of the proof of 
Theorem~\ref{thm-GKPtwo} requires a fair amount of regularity on the boundary. 
Any considerable reduction of the regularity assumptions in 
Theorem~\ref{thm-GKPtwo} would require a new approach or 
a perturbation result for the left hand side of Equation \eqref{eq-thm-2} and \eqref{eq-thm-higherd}.
\end{remark}

The rest of the paper is devoted to the proof of Theorem~\ref{thm-GKPone} and 
Theorem~\ref{thm-GKPtwo}. The proof of Theorem ~\ref{thm-GKPone} is contained 
in Section~\ref{sec:proofone}. The proof of Theorem~\ref{thm-GKPtwo} is divided 
into two sections: the bulk of the proof is contained in 
Section~\ref{sec:prooftwo} 
except for a technical lemma, Lemma~\ref{lem:toeplitztva}, which is proved in 
Section~\ref{sec:jobbigtlemma}. Similarly to~\cite{P,puro}, the main idea in 
both proofs is to compare the resolvent of $L^\tau_{\Omega,b}$ with the 
resolvent of the Landau Hamiltonian.  Roughly speaking, the resolvents are 
compact perturbation of one another. 

Section~\ref{sec:proofone} goes analogously to~\cite[Section $3.1$]{P}. It is 
included for completeness and providing a notational introduction. As mentioned 
above, it relies heavily on the resolvent techniques introduced 
in~\cite[Proposition~2.1]{puro}. 

The asymptotics in Theorem~\ref{thm-GKPtwo} comes from the spectral 
asymptotics of Toeplitz operators on the Landau levels, these deep results were 
proven in \cite{fipu, mero}. We recall them in 
Subsection~\ref{subsec:specofcetrain}. The key step in obtaining the asymptotic 
accumulation of eigenvalues described in Theorem~\ref{thm-GKPtwo}, is the 
reduction to the case of Toeplitz operators from~\cite{fipu, mero} via a 
certain pseudo-differential operator on the boundary. The relevant 
pseudo-differential operators, e.g. boundary layer potentials and Dirichlet to 
Robin operators, are introduced in the 
Subsections~\ref{ressub},~\ref{subsec:boundary} and~\ref{sec:DRRD}. 
The reduction is carried out in Subsection~\ref{subsec:reduction} apart from a 
technical lemma. The technical Lemma~\ref{lem:toeplitztva} states the 
equivalence of the quadratic form associated with a 
certain pseudo-differential operator of order $1$ and the $H^{1/2}$-norm on 
the boundary. It is proven via standard pseudo-differential techniques in 
Section~\ref{sec:jobbigtlemma}.

\section{Proof of Theorem~\ref{thm-GKPone}}
\label{sec:proofone}

In this section we prove Theorem~\ref{thm-GKPone}. As remarked above, the proof 
goes along the lines of~\cite[Section~3.1]{P}. After adding on a Landau 
Hamiltonian $L_{K,b}^{-\tau}$ on $K$ (the sign of $-\tau$ comes 
from the orientation on the boundary) we can consider an operator densely 
defined in $L^2(\R^{2d})$ coinciding in form sense on the form domain of the 
usual Landau Hamiltonian. While $L_{K,b}^{-\tau}$ has discrete spectrum, the 
proof of Theorem~\ref{thm-GKPone} is deduced below in 
Corollary~\ref{corl:proof-mainthm} from the abstract results from~\cite{biso} 
and~\cite{puro} reviewed in the next subsection.

\subsection{Two abstract results}
In this section we state two abstract results.
We will use the first result to conclude positivity of difference of 
resolvents and the second one to obtain the finiteness of eigenvalues above
each Landau level.

\begin{lemma}[Pushnitski-Rozenblum~{\cite[Proposition~2.1]{puro}}]
\label{lem-PR}
Assume that $A$ and $B$ are two self-adjoint positive operators satisfying the 
following hypotheses:
\begin{itemize}
\item $0\not\in \sigma(A)\cup\sigma(B)$.
\item The form domain of $A$ contains that of $B$, i.e.
$D(B^{1/2})\subset D(A^{1/2})$.
\item For all $f\in D(B^{1/2})$, 
$\bigl\|A^{1/2}f\bigr\|=\bigl\|B^{1/2}f\bigr\|$, i.e.
the quadratic forms of $A$ and $B$ agree on the form
domain of $B$. \end{itemize} Then, $B^{-1}\leq A^{-1}$ in the quadratic form
sense, i.e.
\begin{displaymath}
\bigl\langle B^{-1}f,f\bigr\rangle
\leq \bigl\langle A^{-1}f,f\bigr\rangle\quad\forall~f.
\end{displaymath}
\end{lemma}

\begin{lemma}[{\cite[Theorem~9.4.7]{biso}}]
\label{lem-biso}
Assume $A$ is a  self-adjoint operator  and $V$ a compact and positive operator 
such that the spectrum of $A$ in an interval
$(\alpha,\beta)$ is discrete and does not accumulate at $\beta$. Then the 
spectrum of the operator $B=A+V$ in $(\alpha,\beta)$ is discrete and does not 
accumulate at $\beta$. 
\end{lemma}

\subsection{Some facts about the Landau Hamiltonian in $\R^{2d}$}
\label{sec:landau}
In this section we review classical results concerning the Landau Hamiltonian
\begin{equation}\label{eq-LH}
L=-(\nabla-ib\Ab_0)^2\quad\text{in }\mathbb{R}^{2d}\,.
\end{equation}
Here $\Ab_0$ is the magnetic potential of a unit constant magnetic field of 
full rank introduced in~\eqref{mp-A0}, and $b$ is a positive constant. The 
form domain of $L$ is the magnetic Sobolev space
\begin{displaymath}
H^1_{\Ab_0}(\R^{2d})=\bigl\{u\in L^2(\R^{2d})~:~(\nabla-ib\Ab_0)u\in
L^2(\R^{2d})\bigr\}\subseteq H^1_{\text{loc}}(\R^{2d})\,.
\end{displaymath}
The spectrum of $L$ consists of infinitely degenerate eigenvalues
called Landau levels,
\begin{displaymath}
\sigma(L)=\sigma_{\text{ess}}(L)
=\bigl\{\Lambda_q~:~q\in \Nzero\bigr\}.
\end{displaymath}
We denote by $\mathcal L_q$ the eigenspace associated with the
Landau level $\Lambda_q$, i.e.
\begin{equation}\label{eq-Ln}
\mathcal{L}_q=\Ker(L-\Lambda_q)\quad\forall~q\in\Nzero\,.
\end{equation}
We use the notation $P_q$ for the orthogonal projection onto the 
eigenspace $\mathcal L_q$.

The operator $L$ can be expressed in terms of creation and annihilation 
operators. We introduce the complex notation 
$z_j=x_{2j-1}+ix_{2j}$, $j=1,\ldots, d$ and let $\Psi=\frac{b}{4}|z|^2$ be 
a scalar potential for the magnetic field, i.e. $\Delta\Psi=b$. The differential expressions
\begin{equation*}
\overline{Q}_j=-2i e^{-\Psi}\frac{\partial}{\partial z_j}e^\Psi,\quad
Q_j = -2i e^\Psi\frac{\partial}{\partial \overline{z}_j} e^{-\Psi}
\end{equation*}
formally satisfy the following well known identities:
\begin{equation}\label{eq-creann}
\begin{aligned}
\overline{Q}_j&=Q^*_j,\quad 1\leq j\leq d,\\
\bigl[Q_j,\overline{Q}_k\bigr]&=2b\delta_{jk},\quad 1\leq j,k\leq d,\\
L &=\sum_{j=1}^d \overline{Q}_jQ_j+bd.
\end{aligned}
\end{equation}

\subsection{Extension of $L_\Omega$ to an operator in $L^2(\R^{2d})$}
\label{subsectionextensions}

We introduced the operator 
$L_{\Omega,b}^\tau$ with quadratic form $\l_{\Omega,b}^\tau$ 
in~\eqref{QF-q}. We will use also the corresponding operator in $K^\circ$, namely 
$L_{K,b}^{-\tau}$. We will throughout the paper work under the assumption that 
the two quadratic forms $\l_{\Omega,b}^\tau$ and $\l_{K,b}^{-\tau}$ are 
strictly positive. This is always attainable after a shift of the quadratic 
forms by a constant. The effect of the constant is merely a shifting of the 
spectrums of all involved operators, hence we will for notational simplicity 
always assume that this constant is $0$.

\begin{remark}
Notice that, for $u\in H^1_{\Ab_0}(\R^{2d})$,
\[
\l_{\Omega,b}^\tau(u_\Omega)+\l_{K,b}^{-\tau}(u_K)
=\int_{\R^{2d}}|(\nabla-ib\Ab_0)u|^2\,d x.
\]
This motivates the usage of $-\tau$ for the quadratic form on $K$.
For $u_\Omega\in D(L_{\Omega,b}^\tau)$ and $u_K\in D(L_{K,b}^{-\tau})$,
\[
\partial_R u_\Omega=\partial_R u_K=0 \quad \mbox{on}\quad \partial\Omega,
\]
where $\partial_R$ denotes the Robin differential expression 
from~\eqref{eq:dNdR}. 
\end{remark}

When there is no ambiguity, we will skip $b$ and $\tau$ 
from the notation, and write $L_\Omega$, $L_K$, $\l_\Omega$ and $\l_K$ for the 
operators $L_{\Omega,b}^\tau$, $L_{K,b}^{-\tau}$, the quadratic forms 
$\l_{\Omega,b}^\tau$ and $\l_{K,b}^{-\tau}$ respectively.

Since $\Omega$ and $K$ are complementary in 
$\R^{2d}$, the space $L^2(\R^{2d})$ is decomposed as a direct sum 
$L^2(\Omega)\oplus L^2(K)$. This permits us to extend the operator
$L_\Omega$ in $L^2(\Omega)$ to an operator $\widetilde L$ in
$L^2(\R^{2d})$. We let $\widetilde L=L_\Omega\oplus L_K$ in
$D(L_\Omega)\oplus D(L_K)\subset L^2(\R^{2d})$. More precisely,
$\widetilde L$ is the self-adjoint extension associated with the
quadratic form
\begin{equation}\label{eq-QF-q}
\widetilde \l(u)=\l_\Omega(u_\Omega)+\l_K(u_K)\,,\quad u=u_\Omega\oplus
u_K\in L^2(\R^{2d}),\quad u_\Omega\in D(\l_\Omega),\quad u_K\in D(\l_K)\,.
\end{equation}
Since $\l_{\Omega}$ and $\l_{K}$ are strictly positive, we may speak of the 
inverse $\widetilde L^{-1}$ of $\widetilde L$. We have the 
following lemma.

\begin{lemma}
\label{lem:esssp}
With $\widetilde L$ and $L_\Omega$ defined as above:
\begin{enumerate}
\item $\sigma_{\text{ess}}(L_\Omega)
=\sigma_{\text{ess}}(\widetilde L)$.
\item $\lambda\in\sigma_{\text{ess}}\bigl(\widetilde{L}^{-1}\bigr)
\setminus\{0\}$ 
if and only if $\lambda\neq 0$ and 
$\lambda^{-1}\in\sigma_{\text{ess}}(L_\Omega)$.
\end{enumerate}
\end{lemma}

\begin{proof}
Since $\widetilde L=L_\Omega\oplus L_K$, then 
$\sigma(\widetilde L)=\sigma(L_\Omega)\cup\sigma(L_K)$. But $K$ is compact 
and has a Lipschitz boundary, hence $L_K$  has a compact resolvent 
by~\cite[Theorem $1.4.3.2$]{noffnoff}. 
Thus $\sigma_{\text{ess}}(L_K)=\varnothing$ and the first assertion in the 
lemma above follows. Moreover, $L_\Omega$ and $L_K$ are both strictly positive
by hypothesis, hence $0\not\in\sigma(\widetilde L)$. It is straight forward 
that
\[
\sigma_{\text{ess}}(\widetilde
L)=\Bigl\{\lambda\in\R\setminus\{0\}~:~\lambda^{-1}
\in\sigma_{\text{ess}}\bigl(\widetilde{L}^{-1}\bigr)\Bigr\}.
\]
\end{proof}

\subsection{Essential spectrum of $L_\Omega$} 

With the operator $\widetilde L$ introduced above, we can view $L_\Omega$ as a
perturbation of the Landau Hamiltonian $L$ in $\R^{2d}$ introduced in
\eqref{eq-LH}. Actually, we define $V:L^2(\R^{2d})\to L^2(\R^{2d})$ as
\begin{displaymath}
V=\widetilde L^{-1}-L^{-1}.
\end{displaymath}
Then we have the following result on the operator $V$.

\begin{lemma}
\label{lem:V}
The operator $V$ is positive and compact.
Moreover, for all $f,g\in L^2(\R^{2d})$,
\begin{equation}
\label{eq-V}
\langle f,Vg\rangle_{L^2(\R^{2d})}
=\int_{\partial\Omega}\partial_R u\cdot\overline{(v_\Omega-v_K)}\,d S\,,
\end{equation} 
where $u=L^{-1}f$ and $v=\widetilde{L}^{-1} g$.
\end{lemma}

\begin{proof}
Notice that the form domain $H_{\Ab_0}^1(\R^{2d})$ of $L$ is included
in that of $\widetilde L$, and that for $u\in H^1_{\Ab_0}(\R^{2d})$, we
have
\begin{displaymath}
\widetilde \l(u)=\int_{\R^{2d}}|(\nabla-ib\Ab_0)u|^2\,d x\,.
\end{displaymath}
Invoking Lemma~\ref{lem-PR}, we get that the operator $V$ is
positive.

Let us establish the identity in~\eqref{eq-V}. Set $f=Lu$
and $g=\widetilde L v= L_\Omega v_\Omega\oplus L_K v_K$. Then
\begin{displaymath}
\langle f,Vg\rangle_{L^2(\R^{2d})}=\int_\Omega Lu\cdot
\overline{v_\Omega}\,d x+\int_KLu\cdot\overline{v_K}\,d x-\int_\Omega
u\cdot\overline{L_\Omega v_\Omega}\, d x-\int_K
u\cdot\overline{L_Kv_K}\, d x\,.
\end{displaymath}
The identity in~\eqref{eq-V} then
follows by integration by parts and by using the boundary conditions
$\partial_R v_\Omega=\partial_R v_K=0$.

Knowing that the zeroth order trace operators 
$H^{s_1}(K)\to H^{s_2}(\partial \Omega)$ 
and $H^{s_1}_{\text{loc}}(\Omega)\to H^{s_2}(\partial \Omega)$ 
are compact whenever 
$s_1>s_2+1/2>1/2$, cf. \cite[Theorem $9.4$, Chapter $1$]{lions}, 
we conclude from~\eqref{eq-V} that $V$ is a compact operator.
\end{proof}

The localization of $V$ to the boundary carried out in Lemma \ref{lem:V} is by 
now a common triviality used in studying boundary value problems, but a sensation 
around the time of its invention by Birman, see more in \cite{birman61a,birman61b}.
Theorem~\ref{thm-GKPone} follows as a corollary of Lemma~\ref{lem:V}:

\begin{cor}\label{corl:proof-mainthm}
It holds that
\begin{displaymath}
\sigma_{\text{ess}}(L_\Omega)=\bigl\{\Lambda_q~:~q\in\Nzero\bigr\}\,,
\end{displaymath}
and for all $\epsilon\in(0,b)$ and $q\in\Nzero$,
\begin{displaymath}
\sigma(L_\Omega)\cap(\Lambda_q,\Lambda_q+\varepsilon)
\quad\text{is finite.}
\end{displaymath}
\end{cor}

\begin{proof}
Invoking Lemma~\ref{lem:esssp}, it suffices to prove that
\[
\sigma_{\text{ess}}\bigl(\widetilde{L}^{-1}\bigr)\setminus \{0\}
=\bigl\{\Lambda_q^{-1}~:~q\in\Nzero\bigr\}
\]
in order to get the result concerning the essential spectrum of
$L_\Omega$. Notice that $\widetilde{L}^{-1}=L^{-1}+V$ with $V$ a compact
operator. Hence by Weyl's theorem, 
$\sigma_{\text{ess}}\bigl(\widetilde{L}^{-1}\bigr)
=\sigma_{\text{ess}}(L^{-1})$. But we know from
Section~\ref{sec:landau} that $\sigma_{\text{ess}}(L^{-1})\setminus\{0\}
=\bigl\{\Lambda_q^{-1}~:~q\in\Nzero\bigr\}$ as was required to prove.

Since the operator $V$ is compact and positive, invoking Lemma~\ref{lem-biso}, 
we get that 
$\sigma\bigl(\widetilde{L}^{-1}\bigr)
\cap(\Lambda_q^{-1}-\varepsilon,\Lambda_q^{-1})$ is finite. This implies 
that $\sigma(L_\Omega)\cap(\Lambda_q,\Lambda_q+\epsilon)$ is finite.
\end{proof}

\section{Proof of Theorem~\ref{thm-GKPtwo}}
\label{sec:prooftwo}

In this section we prove Theorem~\ref{thm-GKPtwo}. The main idea of the proof is, 
as mentioned above, to reduce the spectral asymptotics in~\eqref{eq-thm-2} 
and~\eqref{eq-thm-higherd} to a similar asymptotics for Toeplitz operators on 
the Landau levels from \cite{fipu,mero}. 
The reduction to Toeplitz operators is by means of localizing to the boundary. 
The localization to the boundary is carried out using the identity~\eqref{eq-V}.

\subsection{The spectrum of certain Toeplitz operators}
\label{subsec:specofcetrain} 

For $q\in\Nzero$ and a measurable set $U\subset\R^{2d}$ the Toeplitz operator
$S_q^U$ is defined by
\begin{equation}\label{eq-toep}
S_q^U=P_q\chi_U P_q\quad \text{in }L^2(\R^{2d})\,.
\end{equation}
Here $\chi_U$ is the characteristic function of $U$. If $U$ is bounded, 
Arzela-Ascoli's theorem implies that $\chi_U P_q$ is a compact operator, 
because Cauchy estimates for holomorphic functions can be generalized to 
the Landau levels. In particular, the Toeplitz operator $S_q^U$ is compact. We 
state the following deep results on these Toeplitz operators. 

\begin{thm}[{\cite[Lemma~3.2]{fipu}}]
\label{lem:toepspectrum}
Assume that $U\subseteq \R^2$ is a bounded domain with Lipschitz boundary. 
Given $q\in\Nzero$, denote by $s_1^{(q)}\geq s_2^{(q)}\geq
\ldots$ the decreasing sequence of eigenvalues of $S^U_q$. Then,
\begin{equation*}
\lim_{j\to{+\infty}}\big(j! s_j^{(q)}\big)^{1/j} =
\frac{b}{2}\left(\kap(U)\right)^2.
\end{equation*}
\end{thm}

\begin{thm}[{\cite[Proposition~7.1]{mero}}]
\label{lem:toepspectrumhigherd}
Assume that $U\subseteq \R^{2d}$ is a bounded domain. Given $q\in\mathbb N$, 
we let $n\bigl(\epsilon,S^U_q\bigr)$ denote the number of eigenvalues of 
$S^U_q$ greater than $\epsilon$. Then
\[
n\bigl(\epsilon,S^U_q\bigr)\sim \begin{pmatrix}q+d-1\\d-1\end{pmatrix} 
\frac{1}{d!}\left(\frac{|\log\epsilon|}{\log|\log \epsilon|}\right)^d
\quad \text{as}\quad\epsilon\to 0^+.
\]
\end{thm}

The reader will recognize the structure of these results and notice that our
main results look very much like them. Indeed, our main task will be to reduce
our situation so that these results can be applied.

\subsection{The resolvent of the Landau Hamiltonian}
\label{ressub}

Since $L$ is strictly positive, $L^{-1}$ is a bounded operator in $L^2(\R^{2d})$
with range $D(L)$. Furthermore, $L^{-1}$ is an operator with an integral
kernel that we denote by $G_0$. This integral kernel is well-known
(see~\cite{si2}) to be
\begin{equation}
\label{gzeroexpression}
G_0(z,\zeta)=\frac{2b^{d-1}}{(4\pi)^d} e^{b(\overline{z}
\cdot \zeta-\overline{\zeta}\cdot z)/4}I\left(\frac{b|z-\zeta|^2}{4}\right),
\end{equation}
where 
\[
I(s)=\int_0^{+\infty} \frac{e^{-s\coth(t)}}{\sinh^d(t)}\,dt.
\]

\begin{remark}
The formula for $G_0(x,y)$ when $d=1$ is more commonly known:
\begin{equation}
\label{eq-kernel}
G_0(x,y)=\int_0^{+\infty} \frac{b}{2\pi\sinh(bs)}
\exp\Bigl(\frac{ib}{2}x\wedge y
-\frac{b}{4\tanh(bs)}|x-y|^2 \Bigr)\,ds,
\end{equation}
where $x\wedge y=x_1y_2-x_2y_1=(\overline{z}\cdot \zeta
-\overline{\zeta}\cdot z)/2$ 
if $z=x_1+ix_2$, $\zeta=y_1+iy_2$.
\end{remark}

\begin{lemma}
\label{lem:kernel}
$L^{-1}$ is an integral operator with kernel $G_0(z,\zeta)$ that has the
 following singularity at the diagonal $z=\zeta$: There exist 
$a_j\in C^\infty(\R^{2d}\times \R^{2d})$, for 
$j\in\Nzero$, and $b_j\in C^\infty(\R^{2d}\times \R^{2d})$, for 
$j\in\Nzero$ with $j\geq \max(1,d-1)$, such that for any large enough $N$, 
there exists a function 
$\tilde{a}_{N+1}\in C^{N}(\R^{2d}\times \R^{2d})$ such that 
\begin{align}
G_0(z,\zeta) &=\frac{1}{2\pi} \log\bigl(\frac1{|z-\zeta|}\bigr)
+\sum_{j=1}^Na_j(z,\zeta)|z-\zeta|^{2j}\nonumber\\
&\quad+\sum_{j=1}^Nb_j(z,\zeta)|z-\zeta|^{2j}\log|z-\zeta|+
\tilde{a}_N(z,z-\zeta), \quad d=1\nonumber\\
G_0(z,\zeta) &=\frac{\Gamma(d-1)}{2\pi^d} |z-\zeta|^{2-2d}
+\sum_{j=1}^Na_j(z,w)|z-\zeta|^{2-2d+2j} \label{eq:as}\\
&\quad+\sum_{j=d-1}^Nb_j(z,\zeta)|z-\zeta|^{2-2d+2j}\log|z-\zeta|+
\tilde{a}_N(z,z-\zeta), \quad d>1.\nonumber
\end{align}
The corresponding expansions, obtained by the term-wise differentiation, 
exist also for $\partial_N G_0(z,w)$ whenever 
$z,w\in \partial \Omega$. Moreover, $G_0(z,w)$ decays as a Gaussian as 
$|z-w|\to{+\infty}$ uniformly in both $z$ and $w$.
\end{lemma}

The proof of this Lemma is of a computational nature and is deferred to 
Appendix~\ref{sec:expansion}, where also the asymptotic expansion is computed 
explicitly.

\begin{remark}
For $d=1$ the integral~\eqref{eq-kernel} can be expressed in terms of the 
Whittaker function
(see~\cite[Section~4.9, formula~(31)]{baer1} and~\cite[Chapter~6]{baer2}) as
\begin{align*}
G_0(x,y)&=\frac{\pi^{3/2}}{b}\Bigl(\frac{b}{8}|x-y|^2\Bigr)^{-3/4}
\exp\Bigl(\frac{ib}{2}x\wedge y\Bigr)\\
&\quad\times\Bigl[W_{\frac12,-\frac12}\Bigl(\frac{b}{2}|x-y|^2\Bigr)
+\frac12W_{-\frac12,-\frac12}\Bigl(\frac{b}{2}|x-y|^2\Bigr)\Bigr].
\end{align*}
Lemma~\ref{lem:kernel} follows in this case from asymptotic formulae for 
Whittaker functions~\cite{baer2}.
\end{remark}

\subsection{Boundary layer operators}
\label{subsec:boundary}

Recall that $K\subset\R^{2d}$ has been assumed to be a compact subset of 
$\R^{2d}$ with smooth boundary and that we defined the domain 
$\Omega=\R^{2d}\setminus K$. Since $\Omega$ 
and $K$ are complementary, the Hilbert space $L^2\bigl(\R^{2d}\bigr)$ is 
naturally decomposed as 
the orthogonal direct sum $L^2(\Omega)\oplus L^2(K)$ in the sense that any 
function $u\in L^2(\R^{2d})$ can be uniquely represented as 
$u_\Omega\oplus u_K$ where $u_\Omega$ and $u_K$ are the restrictions of $u$ 
to $\Omega$ and $K$ respectively.

We let $\Psi^*(\partial\Omega)$ denote the filtered algebra of 
classical\footnote{Also known under the name $1$-step polyhomogeneous.} 
pseudo-differential operators on the common boundary $\partial\Omega$ of 
$\Omega$ and $K$. For a reference 
on pseudo-differential operators, the reader is referred 
to~\cite[Chapter $5$]{agr2},~\cite[Chapter $18$]{horIII} 
or~\cite[Chapter I]{shubin}. We fix a classical pseudo-differential operator
$\tau\in \Psi^0(\partial\Omega)$. The proofs also works for $\tau\in \Psi^t(\partial\Omega)$,
for $t<1$ but becomes notationally more complicated. We restrict 
our attention to the case that $\tau$ is self-adjoint, in the following two subsections
this assumption is needed merely to simplify proofs. 

For the boundary considerations of this Section, more care in the analysis of 
the Robin boundary differential expressions on $\partial \Omega$ defined 
in~\eqref{eq:dNdR} is needed.
We note that by~\cite[Theorem~9.4, Chapter~1]{lions} the magnetic normal 
derivative, cf.~\eqref{eq:dNdR}, gives well defined continuous operators for 
$s\in \R\setminus \{\Z+\frac{1}{2}\}$ such that $s>3/2$:
\[
\partial_{N,\Omega}:H^s_{\text{loc}}(\Omega)\to H^{s-3/2}(\partial \Omega)
\quad\mbox{and}\quad 
\partial_{N,K}:H^s(K)\to H^{s-3/2}(\partial \Omega).
\]
We remind the reader that the normal derivative appearing in 
$\partial_{N,K}$ and $\partial_{N,\Omega}$ are both 
with respect to the unit \emph{outward} normal vector to the boundary of 
$\Omega$. Again following \cite[Theorem~9.4, Chapter~1]{lions}, the trace 
operators 
\[\gamma_{0,\Omega}:H^s_{\text{loc}}(\Omega)\to H^{s-1/2}(\partial \Omega)
\quad\mbox{and}\quad 
\gamma_{0,K}:H^s(K)\to H^{s-1/2}(\partial \Omega),\]
mapping a function to its boundary value, are continuous for $s\in \R\setminus \{\Z+\frac{1}{2}\}$ such that $s>1/2$. 
For $s\in \R\setminus \Z+\frac{1}{2}$ such that $s>3/2$ we define the Robin boundary operators on $\partial \Omega$,
following the expressions ~\eqref{eq:dNdR}, by means of
\begin{gather}
\label{eq-normalder}
\nonumber
\partial_{R,\Omega}:=
\partial_{N,\Omega}+\tau \gamma_{0,\Omega}
:H^s_{\text{loc}}(\Omega)\to H^{s-3/2}(\partial \Omega)
\quad\text{and}\\
\partial_{R,K}:=\partial_{N,K}+\tau \gamma_{0,K}
:H^s(K)\to H^{s-3/2}(\partial \Omega).
\end{gather}
As a rule, we suppress the $K$ and the $\Omega$ from the notation in these 
operators whenever the domain is clear from the context.
We sometimes write $(\partial_R)_x$ in order to stress that the differentiation 
in~\eqref{eq-normalder} is with respect to the variable $x$.

With $G_0(x,y)$ as in~\eqref{gzeroexpression}, we define the operators 
$\mathcal{A}$, $\mathcal{B}$, $A$ and $B$, acting on functions defined on 
$\partial\Omega$, as
\begin{equation}
\label{eq-boundaryop}
\begin{aligned}
\mathcal{A}u(x) &= \int_{\partial\Omega}
G_0(x,y)u(y)\, dS(y),\quad x\in\R^{2d}\setminus \partial\Omega,\\
\mathcal{B}u(x) &= \int_{\partial\Omega}
\bigl[({\partial_N})_y G_0(x,y)\bigr]u(y)\,d S(y),
\quad x\in\R^{2d}\setminus\partial\Omega,\\
Au(x) &= \int_{\partial\Omega}
G_0(x,y)u(y)\,d S(y),\quad x\in\partial\Omega,\\
Bu(x) &= \int_{\partial\Omega}
\bigl[({\partial_N})_y G_0(x,y)\bigr]u(y)\,d S(y),\quad x\in\partial\Omega.
\end{aligned}
\end{equation}

The potentials $\mathcal{A}$ and $\mathcal{B}$ are usually called the
single and double layer potentials. They satisfy $L\mathcal{A}u(x)=0$ and
$L\mathcal{B}u(x)=0$ in $\R^{2d}\setminus{\partial\Omega}$. We will write
limit relations at the boundary for these potentials. 
We refer to~\cite[Chapter 3, Section 12]{agrruss} 
where the corresponding potentials are considered for the Helmholtz 
operator, see also \cite[Section~5.7]{agr2} for the low-dimensional case. 
Since, according to Lemma~\ref{lem:kernel}, the Green functions for both $L$ 
and the Helmholtz operator are globally estimated pseudo-differential 
operators with the same asymptotics in the leading terms as $x-y\to0$, 
the limit relations in~\cite{agr2} apply here as well.
For all $x_0\in\partial\Omega$ it holds that
\begin{equation}
\label{eq:limitrelations}
\begin{gathered}
\begin{aligned}
\lim_{\substack{x\in K^\circ\\x\to x_0}} (\mathcal{A}u)(x)&=(Au)(x_0),\quad
&\lim_{\substack{x\in K^\circ\\x\to x_0}} 
(\mathcal{B}u)(x)&=\frac{1}{2}u(x_0)+(Bu)(x_0),
\\
\lim_{\substack{x\in \Omega\\x\to x_0}} (\mathcal{A}u)(x)&=(Au)(x_0),\quad
&
\lim_{\substack{x\in \Omega\\x\to x_0}} 
(\mathcal{B}u)(x)&=-\frac{1}{2}u(x_0)+(Bu)(x_0),
\end{aligned}\\
\begin{gathered}
\lim_{\substack{x\in \Omega\\x\to x_0}} \partial_N(\mathcal{A}u)(x)
-
\lim_{\substack{x\in K^\circ\\x\to x_0}} \partial_N(\mathcal{A}u)(x)
=u(x_0).\\
\end{gathered}
\end{gathered}
\end{equation}

We recall from~\cite{agrruss,agr2} that any function
$u=u_\Omega\oplus u_K\in H^s_{\text{loc}}(\Omega)\oplus H^s(K)$ 
exponentially decaying 
at $\infty$ and solving $Lu=0$ in $\Omega\dot{\cup} K^\circ$ admits the 
representation by the formulas
\[
\begin{aligned}
u&=\mathcal{B}(\gamma_{0,K}u_K)-\mathcal{A}\bigl(\partial_{N,K} u_K\bigr),
\quad \text{in $K^\circ$,}\\
u&=\mathcal{A}\bigl(\partial_{N,\Omega} u_\Omega)
-\mathcal{B}(\gamma_{0,\Omega}u_\Omega),
\quad \text{in $\Omega$.}
\end{aligned}
\]
Using the limit values~\eqref{eq:limitrelations}, we obtain the following
formulas connecting Dirichlet and Robin data at the boundary $\partial\Omega$:
\begin{equation}
\label{eq:boundaryrelations}
\begin{aligned}
\Bigl(B+A\tau-\frac12\Bigr)\gamma_{0,K}u_K &= A(\partial_{R,K} u_K),\\
\Bigl(B+A\tau+\frac12\Bigr)\gamma_{0,\Omega}u_\Omega 
&= A(\partial_{R,\Omega} u_\Omega).\\
\end{aligned}
\end{equation}

We will use these relations to define the Dirichlet to Robin and Robin to 
Dirichlet operators in Section~\ref{sec:DRRD} below. Before doing so, we present
some results on the operators $A$ and $B$. Recall that for $p\geq1$, 
the symmetrically normed ideal of weak Schatten class operators 
$\mathcal{J}_{p,w}(L^2(\partial\Omega))$ consists of those compact operators
$C$ on $L^2(\partial\Omega)$ whose singular values $\{\mu_k(C)\}_{k\in \N}$ 
behaves like $\mu_k(C)=O(k^{-1/p})$ as $k\to \infty$. See more in~\cite{simon}.

\begin{lemma}
\label{lem:pseudo-zero}
The operators $A$ and $B$ are classical pseudo-differential operators of order 
$-1$. Furthermore:
\begin{enumerate}
\item $A$ defines a self-adjoint operator on $L^2(\partial\Omega)$;
\item A is elliptic with constant principal symbol;
\item $A$ defines a Fredholm operator 
$A:L^2(\partial\Omega)\to H^1(\partial\Omega)$ whose index vanishes;
\item $A$ and $B$, considered as operators on $L^2(\Omega)$, belong to the weak 
Schatten class $\mathcal{J}_{2d-1,w}(L^2(\partial\Omega))$.
\end{enumerate}
\end{lemma}

\begin{proof}
It follows directly from the asymptotics in Lemma~\ref{lem:kernel} that $A$ 
and $B$ are classical pseudo-differential operators of order $-1$, see for instance 
\cite[Theorem 18.2.8]{horIII}. As such, 
the Weyl law on $\partial\Omega$ implies that $A$ and $B$ belong to the weak 
Schatten class $\mathcal{J}_{2d-1,w}(L^2(\partial\Omega))$. It is also clear 
from Lemma~\ref{lem:kernel} that $\sigma_{-1}(A)=\frac{\Gamma(d-1)}{2\pi^d}$. 
Hence $A$ is elliptic. Furthermore, Equation~\eqref{gzeroexpression} implies 
that $G_0(z,w)=\overline{G_0(w,z)}$ so $A$ is self-adjoint on 
$L^2(\partial\Omega)$. 
Since $\sigma_{-1}(A)$ is a constant mapping, it is a lower order perturbation of
an invertible pseudo-differential operator and the statement 
$\ind(A:L^2(\partial \Omega)\to H^1(\partial \Omega))=0$ follows.
\end{proof}

\begin{lemma}
\label{lem:pseudo}
The elliptic operator $A$ defines an isomorphism
$A:L^2(\partial\Omega)\to H^1(\partial\Omega)$.
\end{lemma}

\begin{proof}
To prove that $A$ is an isomorphism, we follow the proof 
of~\cite[Chapter~7, Proposition~11.5]{tay2} with the necessary modifications.
Since the index of $A:L^2(\partial\Omega)\to H^1(\partial\Omega)$ vanishes, 
it suffices to prove that the operator $A$ is injective. By elliptic regularity, it suffices 
to prove that $A:C^\infty(\partial\Omega)\to C^\infty(\partial\Omega)$ is 
injective.

Assume that $h\in C^\infty(\partial\Omega)$ with $Ah=0$. If we define 
$u\in C^\infty(K^\circ)$ by $u(x)=\mathcal{A}h(x)$, $x\in K^\circ$, then $u$ satisfies
\begin{equation*}
\begin{cases}
-(\nabla-ib\mathbf{A}_0)^2u = 0 & \text{in $K^\circ$}, \\
u = 0 & \text{on $\partial\Omega$.}
\end{cases}
\end{equation*}
We use~\eqref{eq-creann} and integrate by parts, to get
\begin{equation*}
0 = \bigl\langle -(\nabla-b\mathbf{A}_0)^2 u,u\bigr\rangle_{L^2(K)} 
= bd\|u\|_{L^2(K)}^2
+ \sum_{j=1}^{d}\|Q_ju\|_{L^2(K)}^2.
\end{equation*}
This implies that $u\equiv 0$ in $K$, i.e.
\begin{equation}
\label{eq-Ah0}
\mathcal{A}h(x)\equiv 0\quad \text{in $K^\circ$.}
\end{equation}

It follows from the limit relations~\eqref{eq:limitrelations} that
$\partial_N(\mathcal{A}h)(x)$ makes
a jump across the boundary $\partial\Omega$ of size $h$, so if we let
$w(x)=\mathcal{A}h(x)$, $x\in\Omega$, then it satisfies
\begin{equation}\label{eq-outer}
\begin{cases}
-(\nabla-ib\mathbf{A}_0)^2w = 0 & \text{in $\Omega$}, \\
\partial_N w = h & \text{on $\partial\Omega$.}
\end{cases}
\end{equation}
Since, again by~\eqref{eq:limitrelations}, $\mathcal{A}h$ does not jump 
across $\partial\Omega$, we see 
by~\eqref{eq-Ah0} that $w=0$ on $\partial\Omega$.

From the exponential decay of $G_0(x,y)$ as $|x-y|\to{+\infty}$ it follows that
$w(x)=O(|x|^{-N})$ as $|x|\to{+\infty}$ for all $N>0$. This also applies to all
derivatives of $w$.
Moreover $w$ is smooth. Hence we can integrate by parts in $\Omega$ to find
\begin{equation*}
0 = \bigl\langle -(\nabla-ib\mathbf{A}_0)^2 w,w\bigr\rangle_{L^2(\Omega)} 
= bd\|w\|_{L^2(\Omega)}^2 + \sum_{j=1}^{d}\|Q_jw\|_{L^2(\Omega)}^2,
\end{equation*}
and hence $w\equiv 0$ in $\Omega$. 
From~\eqref{eq-outer} we see that $h=0$ on $\partial\Omega$.
\end{proof}

\subsection{The Dirichlet to Robin and Robin to Dirichlet operators}
\label{sec:DRRD}

Let $\phi\in L^2(\partial\Omega)$ be given, and let $u$ be a solution with exponential 
decay at infinity to the exterior Robin problem
\[
\begin{cases}
Lu=0,& \text{in $\Omega$},\\
\partial_R u=\phi, & \text{on $\partial\Omega$}.
\end{cases}
\]
We will see below in Equation~\eqref{eq:RDDR} that the existence of 
$u$ is guaranteed in a subspace of finite codimension. This solution 
is unique, provided certain orthogonality conditions are imposed. 
We denote by $\RDO\phi$ the boundary values of $u$ at $\partial\Omega$, 
whenever $\phi$ admits a solution $u$.
The operator $\RDO:\phi\mapsto\RDO\phi$ is called the \emph{exterior Robin to
Dirichlet} operator for the differential equation $Lu=0$. We define the
\emph{interior Robin to Dirichlet operator} $\RDK$ in a similar way. Their
inverse operators, associating Robin data of solutions to their Dirichlet data
are called the \emph{exterior and interior Dirichlet to Robin} operators, and
are denoted $\DRO$ and $\DRK$ respectively. 

Using the relations in~\eqref{eq:boundaryrelations} we find that these operators 
in fact are independent on the choice of solution $u$ for $\phi$ outside a 
finite-dimensional subspace. It follows from~\eqref{eq:boundaryrelations} that:
\begin{equation}
\label{eq:RDDR}
\begin{aligned}
\Bigl(B+A\tau-\frac12\Bigr)\RDK &= A,\quad\mbox{on} \quad A^{-1}\Bigl(B+A\tau-\frac12\Bigr)C^\infty(\partial\Omega);\\
\Bigl(B+A\tau+\frac12\Bigr)\RDO &= A,\quad\mbox{on} \quad A^{-1}\Bigl(B+A\tau+\frac12\Bigr)C^\infty(\partial\Omega);\\
\DRK &= A^{-1}\Bigl(B+A\tau-\frac12\Bigr);\\
\DRO &= A^{-1}\Bigl(B+A\tau+\frac12\Bigr).
\end{aligned}
\end{equation}
These equations determine $\RDK$, $\RDO$, $\DRK$ and $\DRO$ outside a 
finite-dimensional subspace while $B+A\tau\pm \frac12$ are elliptic 
pseudo-differential operators of order $0$. After a choice of extension, 
Equation \eqref{eq:RDDR} allows us to consider the operators $\RDK$, 
$\RDO$, $\DRK$ and $\DRO$ as pseudo-differential operators on $\partial\Omega$.
More precisely, we have the following standard result for Dirichlet to Robin operators. 
The proof can be found in~\cite[Appendix C of Chapter $12$]{tay2}.

\begin{prop}
The interior and exterior Robin to Dirichlet and Dirichlet to Robin operators
are given by the relations in~\eqref{eq:RDDR}, and are elliptic 
pseudodifferential operators with constant principal symbols:
\[
\RDK,\RDO\in \Psi^{-1}(\partial\Omega)\quad\mbox{and}
\quad \DRK,\DRO\in \Psi^1(\partial\Omega).
\]
\end{prop}

Very oftenly, Equation \eqref{eq:RDDR} determines the Dirichlet to Robin operators; 
the operators $B+A\tau\pm \frac12$ are, in a sense made precise below, 
\emph{generically} invertible. This fact is based on the following lemma.

\begin{lemma}
\label{smallperturbation}
Assume that $M$ is a closed manifold and that $s>0$. Let $T\in \Psi^0(M)$ be 
an elliptic operator of index $0$ and $A\in \Psi^{-s}(M)$ be an invertible 
operator $A:L^2(M)\to H^s(M)$. Viewing $A^{-1}$ and $A^{-1}T$ as unbounded 
operators on $L^2(M)$, then for any $\epsilon\notin -\sigma(A^{-1}T)$, the 
bounded operator
\[
T+\epsilon A:L^2(M)\to L^2(M)
\]
is invertible.
\end{lemma}

The statement of the Lemma makes sense because $T$ is a pseudo-differential 
operator of order zero and as such it preserves the domain $H^s(M)$ of $A^{-1}$.
We also note that by elliptic regularity, the set $-\sigma(A^{-1}T)$ is the 
same when changing 
the domain and range of $A^{-1}T$ to $H^t(M)$ for any $t$.

\begin{proof}
As $\ind(T)=0$ and $A$ is of negative order, the operator 
$T+\epsilon A$ is an elliptic pseudo-differential operator of order $0$. 
As such, it defines an operator $H^t(M)\to H^t(M)$ with index $0$ for 
any $\epsilon$ and $t$. Hence $T+\epsilon A$ is invertible if and only if 
$\ker(T+\epsilon A)=0$. However, $A$ is invertible so $\ker(T+\epsilon A)=0$ 
holds if and only if $\ker(A^{-1}T+\epsilon)=0$. By elliptic regularity these 
subspaces 
do not depend on the choice of domain in the Sobolev scale. The operator 
$A^{-1}T$ is an elliptic pseudo-differential operator of order $s$ with index 
$0$. It follows that $\ind(A^{-1}T+\epsilon)=0$. Thus 
$\ker(A^{-1}T+\epsilon)=0$ holds if and only if 
$-\epsilon\notin \sigma(A^{-1}T)$. 
\end{proof}

We say that $\tau$ is \emph{generic} if the operators 
\[
T_{\pm,\tau}:=B+A\tau\pm \frac12
\]
are invertible. Note the following well-known consequence of elliptic regularity and 
the Fredholm property of elliptic operators, an elliptic pseudo-differential operator 
$D$ of order $m$ is invertible as an operator between Sobolev spaces 
$D:H^s(\partial\Omega)\to H^{s-m}(\partial\Omega)$, for some $s\in \R$, 
if and only if $D$ is invertible inside the algebra of pseudo-differential operators.
The following Corollary motivates the terminology \emph{generic}.

\begin{cor}
\label{genericgamma}
Letting $\tau$, $A$ and $B$ be as above. For all $\epsilon\in [-1,1]$ outside a 
finite subset, $\tau+\epsilon$ is generic.
\end{cor}

\begin{proof}
By \cite[Theorem I.$8.4$]{shubin}, if the spectrum 
$\sigma(A^{-1}T_{\pm,\tau})$ 
is not equal to $\mathbb{C}$, it is a discrete subset of $\mathbb{C}$ and 
$[-1,1]\cap \sigma(-A^{-1}T_{+,\tau})\cap \sigma(-A^{-1}T_{-,\tau})$ is a 
finite set. 
The Corollary follows from Lemma~\ref{smallperturbation} provided there exists 
$\lambda_\pm\in \mathbb{C}$ such that $A^{-1}T_{\pm,\tau}+\lambda_\pm$ are 
invertible. 
We note that the principal symbols
$\pm \sigma_1(A^{-1}T_{\pm,\tau})$ are positive 
constant functions on $S^*\partial \Omega$. Existence of 
$\lambda_\pm\in \mathbb{C}$ follows 
from the G\aa rding inequality (\cite[Theorem $18.1.14$]{horIII}).  
\end{proof}

We turn our attention to associating Robin data on $\partial\Omega$ to
Dirichlet data for functions in the Landau subspace $\mathcal{L}_q$, or 
more generally to solutions of the homogeneous equation 
$(L-\Lambda_q)u=0$ in $K^\circ$.
The construction is well known and can be found in, for 
instance,~\cite[Chapter~7.12 and Appendix~C of Chapter~12]{tay2}.

Let $\mathcal{Q}_q^K\subseteq C^\infty(\overline{K})$ denote the space of 
solutions $u\in C^\infty(\overline{K})$ to $(L-\Lambda_q)u=0$ in $K^\circ$ and 
let $\mathcal{Q}_q^{\partial\Omega}\subseteq C^\infty(\partial\Omega)$ be 
the image 
of $\mathcal{Q}_q^K$ under the restriction mapping $C^\infty(\overline{K})\to 
C^\infty(\partial\Omega)$. Since $L-\Lambda_q$ is a strongly elliptic operator 
in $K$, the space 
$\mathcal{Q}_q^{\partial\Omega}\subseteq C^\infty(\partial\Omega)$
has finite codimension. Since the eigenvalue multiplicities of $L$ in $K$ 
equipped with Dirichlet conditions on $\partial\Omega$ are finite, the kernel of 
$\mathcal{Q}_q^K\to \mathcal{Q}_q^{\partial\Omega}$ is finite-dimensional.

This means that we can, for any function $\phi\in \mathcal{Q}_q^{\partial\Omega}$, solve
\[
\begin{cases}
(L-\Lambda_q)u = 0,&\text{in $K^\circ$,}\\
u = \phi,& \text{on $\partial\Omega$.}
\end{cases}
\]
The condition that $u$ is $L^2$-orthogonal to 
$\ker\bigl(\mathcal{Q}_q^K\to \mathcal{Q}_q^{\partial\Omega}\bigr)$ 
guarantees a unique solution. Define the corresponding solution operator 
\[
^0\DRq:\mathcal{Q}_q^{\partial\Omega}\to C^\infty(\partial\Omega),
\quad ^0\DRq\phi:=\partial_{R,K}u.
\]
We let $\DRq:C^\infty(\partial\Omega)\to C^\infty(\partial\Omega)$ denote 
any extension of this operator, which exists since 
$\mathcal{Q}_q^{\partial\Omega}\subseteq C^\infty(\partial\Omega)$
has finite codimension. The following result is standard and follows 
from~\cite[Chapter $7$]{tay2}.

\begin{lemma}
\label{anelliptic}
The operator $\DRq$ possesses the following properties:
\begin{enumerate}
\item $\DRq$ is an elliptic pseudo-differential operator of order $1$ with 
constant positive principal symbol.
\item For any $s$, $\ind\bigl(\DRq:H^s(\partial\Omega)
\to H^{s-1}(\partial\Omega)\bigr)=0$. 
\item There exists a number $c=c(\partial\Omega,q)\in \R$ such that 
$\DRq:H^{1/2}(\partial\Omega)\to H^{-1/2}(\partial\Omega)$ is 
invertible as an operator on $L^2(\partial\Omega)$ as long as $\tau\geq c$.
\end{enumerate}
\end{lemma}

\begin{proof}
The computation of the principal symbol of $\DRq$ can be found in 
Proposition~C.1 of~\cite[Chapter $12$]{tay2}. The identity $\ind\bigl(\DRq:H^s(\partial\Omega)
\to H^{s-1}(\partial\Omega)\bigr)=0$ follows in the same way as in the proof of 
Lemma \ref{lem:pseudo-zero}. From the index computation for $\DRq$ we conclude that 
$\DRq$ is an isomorphism if and only if 
$\DRq$ is injective. It follows from the G\aa rding 
inequality that there exists a positive constant $c$
such that for $\tau\geq 0$, for a certain constant $c'_\tau>0$,
\[
\Real\bigl\langle  \DRq u,u\bigr\rangle
\geq c'_\tau\|u\|_{H^{1/2}(\partial\Omega)}^2
-c\|u\|_{L^2(\partial\Omega)}^2.
\]
Hence, as long as $\tau\geq c$, we have
$\Real\langle  \DRq u,u\rangle
\geq c'_{\tau-c}\|u\|_{H^{1/2}(\partial\Omega)}^2$. 
Injectivity of $\DRq:H^{1/2}(\partial\Omega)\to H^{-1/2}(\partial\Omega)$ assuming $\tau\geq c$ follows.
\end{proof}

\begin{remark}
\label{DvNvsnormal}
We note that by construction, $\DRq$ coincides with $\partial_{R,K}$ 
outside a finite-dimensional subspace. Hence, for $q\in\Nzero$, a finite rank 
smoothing operator $S_q\in \Psi^{-\infty}(\partial\Omega)$ exists,
such that, as long as $u\in C^\infty(\overline{K})$ satisfies $(L-\Lambda_q) u=0$ in $K^\circ$, then
\[
\partial_{R,K} u
=(\DRq+S_q)\gamma_{0,K}u.
\]
\end{remark}

\subsection{Reduction to a Toeplitz operator}
\label{subsec:reduction}

In this subsection we prove Theorem~\ref{thm-GKPtwo} modulo a technical 
Lemma that we prove in the next section. We assume as above that 
$K\subset \R^{2d}$ is compact with smooth boundary upon which $\tau$ is a classical
pseudo-differential operator which we for simplicity assume to have order $0$.
In the previous subsections, we made the assumption that $\tau$ was self-adjoint to 
simplify proofs while in this subsection it is necessary for the results to hold. 
Let $q\in\Nzero$ and pick $\delta>0$ such that
\[
\Big(\bigl(\Lambda_q^{-1}-2\delta,\Lambda_q^{-1}+2\delta\bigr)
\setminus\bigl\{\Lambda_q^{-1}\bigr\}\Big)\cap
\sigma_{\text{ess}}(\widetilde{L}^{-1})=\varnothing. 
\]
Denote by
$\bigl\{r_j^{(q)}\bigr\}_{j\geq 1}$ the decreasing sequence of eigenvalues of
$\widetilde{L}^{-1}$ in the interval
$(\Lambda^{-1}_q,\Lambda_q^{-1}+\delta)$. 
For each $q\in\Nzero$ we introduce the operator 
\begin{equation}
\label{eq-Tn}
T_q=P_qVP_q\,,
\end{equation} 
where, as before, $P_q$ is the orthogonal projection onto the eigenspace 
$\mathcal L_q$ associated with $\Lambda_q$ and $V=\widetilde L^{-1}-L^{-1}$. 
By Lemma~\ref{lem:V}, $V$ is a 
positive and compact operator. These properties are inherited by $T_q$. 
Denote by $\{t_j^{(q)}\}$ the decreasing sequence of eigenvalues of $T_q$.
The next lemma, proved in \cite[Proposition~2.2]{puro}, shows that
$r_j^{(q)}-\Lambda_q^{-1}$ are close to the eigenvalues of $T_q$.

\begin{lemma}[{\cite[Proposition~2.2]{puro}}]
\label{lem:RT}
Given $\epsilon>0$ there exist integers $l$ and $j_0$ such that
\begin{equation*}
 (1-\epsilon)t_{j+l}^{(q)}\leq r_{j}^{(q)}-\Lambda_q^{-1} 
\leq (1+\epsilon)t_{j-l}^{(q)}, \quad\forall~j\geq j_0.
\end{equation*}
\end{lemma}

The spectrum of $T_q$ will be related further to the spectrums of Toeplitz 
operators for generic operators $\tau$. Recall that
given a compact domain $U\subset\R^{2d}$, we introduced 
in~\eqref{eq-toep} the Toeplitz operator $S_q^U$. We will prove now 
the following result.

\begin{lemma}
\label{lem:toeplitztva}
For all $q\in\Nzero$ there exists a finite-dimensional subspace 
$\mathcal{W}_q\subseteq \mathcal{L}_q$ such that if 
$K_0\subset K \subset K_1$ are 
compact domains with ${\partial K_i\cap \partial K=\varnothing}$ (for $i=0$ 
and $i=1$) there exists a constant $C>1$ such
that
\begin{equation}
\label{eq:inequality}
 \frac1C \langle f,S_q^{K_0}f\rangle_{L^2(\R^{2d})} 
\leq \langle f, T_q f \rangle_{L^2(\R^{2d})} 
\leq C \langle
 f,S_q^{K_1}f\rangle_{L^2(\R^{2d})}\quad\forall~f\in \mathcal{L}_q\ominus\mathcal{W}_q\,.
\end{equation}
\end{lemma}

The proof of Lemma~\ref{lem:toeplitztva} is by reduction of the operator 
$T_q$ to a pseudo-differential operator on the common boundary 
$\partial\Omega$ of $\Omega$ and $K$. We postpone the proof to 
Section~\ref{sec:jobbigtlemma} below, and continue instead with the proof 
of~\eqref{eq-thm-2}.

\begin{cor}
\label{cor:thm-resolvent}
Whenever $K\subset \R^2$ is compact with $C^\infty$-boundary,
\[\lim_{j\to{+\infty}}\Bigl(j!\bigl(r_j^{(q)}-\Lambda_q^{-1}\bigr)\Bigr)^{1/j}
=\frac{b}2\bigl(\kap(K)\bigr)^2\,.\]
In particular,~\eqref{eq-thm-2} holds true.
\end{cor}

\begin{proof}
Invoking the variational min-max principle, the result of 
Lemma~\ref{lem:toeplitztva} provides us with a sufficiently large
integer $j_0\in\mathbb N$ such that, for all $j\geq j_0$, we have,
\begin{displaymath}
\frac1C s_{j,K_0}^{(q)}\leq t_j^{(q)}\leq C s_{j,K_1}^{(q)}\,.
\end{displaymath}
Here $\{s_{j,K_0}^{(q)}\}_j$ and $\{s_{j,K_1}^{(q)}\}_j$ are the decreasing 
sequences of eigenvalues of $S_q^{K_0}$ and $S_q^{K_1}$ respectively.
Applying the result of Theorem~\ref{lem:toepspectrum} in the inequality 
above, we get
\begin{displaymath}
\frac{b}2\left(\text{Cap}(K_0)\right)^2
\leq\lim_{j\to{+\infty}}\left(j!t_j^{(q)}\right)^{1/j}
\leq\frac{b}2\left(\text{Cap}(K_1)\right)^2.
\end{displaymath}
Since both $K_0\subset K\subset K_1$ are arbitrary, we get by making them 
close to $K$,
\begin{displaymath}
\lim_{j\to{+\infty}}\left(j!t_j^{(q)}\right)^{1/j}=\frac{b}2
\left(\text{Cap}(K)\right)^2\,.
\end{displaymath}
Applying the above asymptotic limit in the estimate of Lemma~\ref{lem:RT}, 
we get the announced result in Corollary~\ref{cor:thm-resolvent} above.
\end{proof}

\begin{cor}
\label{cor:thm-resolventhigherd}
Equation~\eqref{eq-thm-higherd} holds true.
\end{cor}

\begin{proof}
This is clear from Lemma~\ref{lem:toeplitztva} and 
Theorem~\ref{lem:toepspectrumhigherd}.
\end{proof}

Summing up the results of 
Corollaries~\ref{corl:proof-mainthm},~\ref{cor:thm-resolvent} 
and~\ref{cor:thm-resolventhigherd}, we end up with the proof of 
Theorem~\ref{thm-GKPtwo}. All that remains is to prove 
Lemma~\ref{lem:toeplitztva}. That will be the subject of the next section.

\section{Proof of Lemma~\ref{lem:toeplitztva}}
\label{sec:jobbigtlemma}

The aim of this section is to prove Lemma~\ref{lem:toeplitztva}. 
The proof in this subsection goes along similar lines as in Subsection~4.2 
of~\cite{P}. Recall the operators $A$ and $B$ from~\eqref{eq-boundaryop}.

\begin{lemma}
\label{parametrixlemma}
Consider the elliptic operators 
$T_{\pm,\tau}=B+A\tau\pm \frac12
\in \Psi^0(\partial\Omega)$. There exist elliptic operators 
$R_{\pm,\tau}\in \Psi^0(\partial\Omega)$, with principal symbol 
$\sigma_0(R_{\pm,\tau})=\pm 2$, such that the operators
\[
T_{\pm,\tau} R_{\pm,\tau}-1\in \Psi^{-\infty}(\partial\Omega)
\quad\text{and}\quad R_{\pm,\tau} T_{\pm,\tau} -1
\in \Psi^{-\infty}(\partial\Omega)
\]
are of finite rank.
\end{lemma}

Lemma \ref{parametrixlemma} could be considered folklore.
In lack of a reference we provide a proof of its statement.

\begin{proof}
We let $\Psi^{-\infty}_{\text{fin}}(\partial\Omega)\subseteq \Psi^0(\partial\Omega)$ 
denote the ideal of finite rank smoothing operators. 
We consider the unital algebras 
\[
\mathfrak{A}_{\text{fin}}:=\Psi^0(\partial\Omega)/\Psi^{-\infty}_{\text{fin}}(\partial \Omega)
\quad\mbox{and}\quad 
\mathfrak{A}:=\Psi^0(\partial\Omega)/\Psi^{-\infty}(\partial \Omega).
\]
There is a quotient mapping $\mathfrak{A}_{\text{fin}}\to  \mathfrak{A}$ 
whose kernel is $\Psi^{-\infty}(\partial\Omega)/\Psi^{-\infty}_{\text{fin}}(\partial\Omega)$. 
It follows by means of the standard techniques of 
pseudo-differential operators that an element $T\in \Psi^0(\partial\Omega)$ 
is elliptic if and only if the equivalence class $T\mod \Psi^{-\infty}(\partial\Omega)$ 
is an invertible element of $\mathfrak{A}$. Hence the Lemma follows if we can 
prove that if $\tilde{a}\in  \mathfrak{A}_{\text{fin}}$ satisfies that 
$a:=\tilde{a}\mod \Psi^{-\infty}(\partial\Omega)/\Psi^{-\infty}_{\text{fin}}(\partial\Omega) \in \mathfrak{A}$ 
is invertible, then so is $\tilde{a}$. 

We choose a lift $\tilde{r}\in \mathfrak{A}_{\text{fin}}$ of $a^{-1}$ 
and consider the elements in $\Psi^{-\infty}(\partial\Omega)/\Psi^{-\infty}_{\text{fin}}(\partial\Omega)$
defined by
\[
s_L:=\tilde{r} \tilde{a}-1_{\mathfrak{A}_{\text{fin}}}
\quad\mbox{and}\quad  
s_R:=\tilde{a} \tilde{r}-1_{\mathfrak{A}_{\text{fin}}}.
\]
We can lift $s_L$ and $s_R$ to smoothing operators $S_L$, $S_R\in \Psi^{-\infty}(\partial\Omega)$. 
The operators $1+S_L$ and $1+S_R$ are elliptic operators of index $0$. Fredholm operators of 
index $0$ are invertible modulo finite rank operators. If follows by elliptic regularity that $1+S_L$ 
and $1+S_R$ are invertible modulo finite rank smoothing operators. Hence $1_{\mathfrak{A}_{\text{fin}}}+s_L$ 
and $1_{\mathfrak{A}_{\text{fin}}}+s_R$ are invertible elements of $\mathfrak{A}_{\text{fin}}$.

Let us define $\tilde{r}_R:=\tilde{r}(1_{\mathfrak{A}_{\text{fin}}}+s_R)^{-1}$ and 
$\tilde{r}_L:=(1_{\mathfrak{A}_{\text{fin}}}+s_L)^{-1}\tilde{r}$.
A direct computation shows that 
\[
\tilde{a}\tilde{r}_R=\tilde{r}_L\tilde{a}=1_{\mathfrak{A}_{\text{fin}}}.
\]
Multiplying the identity $\tilde{a}\tilde{r}_R=1_{\mathfrak{A}_{\text{fin}}}$ with $\tilde{r}_L$ from the left 
proves that $\tilde{r}_R=\tilde{r}_L$.  It follows that $\tilde{a}$ is invertible with inverse $\tilde{r}_L$.
\end{proof}

By Lemma~\ref{smallperturbation}, 
it is generically the case that $T_{\pm,\tau}$ is invertible.
To simplify notation we set 
$\hat{F}_{\pm,\tau}:=R_{\pm,\tau} T_{\pm,\tau} -1$.

\begin{lemma}
\label{lem:Tn}
Let $q\in\Nzero$. There exists a finite rank smoothing operator 
$F_q:L^2(\R^{2d})\to L^2(\partial\Omega)$ such that for 
all $f,g\in L^2(\R^{2d})$,
\begin{equation}\label{eq-lem:Tn}
\langle f,T_q g\rangle_{L^2(\R^{2d})}
=\frac{1}{\Lambda_q^2}\int_{\partial\Omega}
(P_qf)\cdot\overline{T_{\tau,q}(P_q g)}\, d S
+\langle \gamma_0 P_qf,F_qg\rangle_{L^2(\partial\Omega)}\,,
\end{equation}
where $T_{\tau,q}$ is the elliptic operator defined by
\[
T_{\tau,q}:=\left(\DRq+S_q\right)^*(R_{-,\tau}-R_{+,\tau})A\left(\DRq
+S_q\right) \in \Psi^1(\partial\Omega).
\]
\end{lemma}

\begin{proof}
We set $u=L^{-1}P_qf=\Lambda_q^{-1}P_qf$,
$v=\widetilde{L}^{-1} P_qg=v_\Omega\oplus v_K$ and
$w=L^{-1}P_qg=\Lambda_q^{-1}P_qg$. Notice that
\begin{displaymath}
\langle f,T_q g\rangle_{L^2(\R^{2d})}
=\langle P_q f,V P_q g\rangle_{L^2(\R^{2d})}
\end{displaymath}
where $V$ is the operator defined in~\eqref{eq-V}. 
Invoking Lemma~\ref{lem:V}, we 
write,
\begin{align*}
\langle P_q f,V P_q g\rangle_{L^2(\R^{2d})} 
&= \int_{\partial\Omega}\partial_R u
\cdot
\overline{(v_\Omega-v_K)}\, d S\\
&=\int_{\partial\Omega}\partial_R u
\cdot
\overline{(v_\Omega-w+w-v_K)}\, d S.
\end{align*}
Note that since $u\in H^1_{\Ab_0}(\R^{2d})$, $\partial_R u=\partial_{R,K}u
=\partial_{R,\Omega}u$. Using~\eqref{eq:RDDR} and 
Lemma~\ref{parametrixlemma}, we can write further,
\begin{align}
\label{firstandsecond}
\langle P_q f, V P_q g\rangle_{L^2(\R^{2d})}
&=\int_{\partial\Omega}
\partial_R u
\cdot
\overline{(R_{+,\tau}A(\partial_{R,\Omega}(v_\Omega-w))+R_{-,\tau}
A(\partial_{R,K}(w-v_K)))}\, d S\\
\nonumber
&\quad-\int_{\partial\Omega}
\partial_R u
\cdot
\overline{\hat{F}_{+,\tau}(v_\Omega-w)
+\hat{F}_{-,\tau}(w-v_K))}\, d S.
\end{align}
Notice that $v_\Omega$ and $v_K$ are in the domain of the operators
$L_\Omega$ and $L_K$ respectively, hence 
$\partial_R v_\Omega=\partial_R v_K=0$. By construction, $L_qu=L_qw=0$ in 
$K$, and Remark~\ref{DvNvsnormal} implies that 
\begin{equation}
\label{magicalformulas}
\partial_R u=(\DRq+S_q)
\gamma_0u
\quad\mbox{and} \quad 
\partial_R w
=(\DRq+S_q)\gamma_0w,
\end{equation}
for some finite rank smoothing operator 
$S_q\in \Psi^{-\infty}(\partial\Omega)$. 
Consequently, after applying this identity to the first term 
in~\eqref{firstandsecond} we get
\begin{align*}
\label{eq-lem-RD}
\langle P_qf,&VP_qg\rangle_{L^2(\R^{2d})}
+\int_{\partial\Omega}
\partial_R u
\cdot
\overline{\hat{F}_{+,\tau}(v_\Omega-w)
+\hat{F}_{-,\tau}(w-v_K))}\,d S\\
&=\int_{\partial\Omega}(\DRq+S_q)u
\cdot
\overline{(R_{-,\tau}-R_{+,\tau})A(\DRq+S_q)w}\,d S
=\langle u,T_{\tau,q}w\rangle_{L^2(\partial \Omega)}.
\end{align*} 

As for the second term in~\eqref{firstandsecond},
\[
\begin{multlined}
\int_{\partial\Omega}\partial_R 
u
\cdot\overline{\hat{F}_{+,\tau}(v_\Omega-w)
+\hat{F}_{-,\tau}(w-v_K))}\, d S\\
=\int_{\partial\Omega}  
u
\cdot
\overline{\underbrace{(\DRq+S_q)^*(\Lambda_q^{-1}
(\hat{F}_{-,\tau}-\hat{F}_{+,\tau})+(\hat{F}_{+,\tau}L_\Omega^{-1}
-\hat{F}_{-,\tau}L_K^{-1}))P_q}_{\Lambda_q F_q}g}\, d S.
\end{multlined}
\]
Since $\hat{F}_{\pm,\tau}$ are of finite rank, it is clear that $F_q$ is of 
finite rank.
\end{proof}

\begin{lemma}
\label{controllingnonpositive}
The operator $T_{\tau,q}$ has discrete spectrum, and there exists a finite rank 
operator $S_{\tau,q}\in \Psi^{-\infty}(\partial\Omega)$ such that for some 
$b,C>0$,
\[\|\phi\|_{H^{1/2}(\partial\Omega)}
\leq C\left(\Real\langle\phi, T_{\tau,q}\phi\rangle_{L^2(\partial\Omega)}
+b\|S_{\tau,q}\phi\|_{L^2(\partial\Omega)}\right).\]
\end{lemma}

\begin{proof}
By an argument similar to that of Corollary~\ref{genericgamma}, it follows from 
\cite[Theorem I.$8.4$]{shubin} and the G\aa rding inequality that 
$T_{\tau,q}$ has discrete spectrum. We define the elliptic self-adjoint 
first order 
pseudo-differential operator
\[
\tilde{T}_{\tau,q}:=\frac12(T_{\tau,q}+T_{\tau,q}^*).
\]
Since $T_{\tau,q}$ has positive principal symbol, we see that 
$T_{\tau,q}-\tilde{T}_{\tau,q}\in \Psi^0(\partial\Omega)$ defines a bounded 
operator. The G\aa rding inequality implies that for some $b,C>0$,
\[\|\phi\|_{H^{1/2}(\partial\Omega)}\leq 
C\left(\langle \phi,\tilde{T}_{\tau,q}\phi\rangle_{L^2(\partial\Omega)}+
b\|\phi\|_{L^2(\partial\Omega)}\right).\]
Since $\tilde{T}_{\tau,q}$ is of order $1$, elliptic and self-adjoint, its 
spectrum is a discrete subset of $\R$. 
The G\aa rding inequality implies that 
$\tilde{T}_{\tau,q}$ 
is bounded from below, so the spectrum of $\tilde{T}_{\tau,q}$ only accumulates 
at $+\infty$ and there are only finitely many non-positive eigenvalues. 
We define $S_{\tau,q}\in \Psi^{-\infty}(\partial\Omega)$ as the finite rank 
projection onto the non-positive eigenspace of $\tilde{T}_{\tau,q}$. For a, 
possibly new, constant $C$, the Lemma follows because 
\[
\Real\langle \phi, T_{\tau,q}\phi\rangle_{L^2(\partial\Omega)}
=\langle \phi, \tilde{T}_{\tau,q}\phi\rangle_{L^2(\partial\Omega)}.
\]
\end{proof}

\begin{prop}
\label{fredholmofl}
The linear operator
\begin{equation}
\label{hetttillhminett}
H^1(K)\to H^{-1}(K)\oplus H^{1/2}(\partial\Omega), 
\quad f\mapsto (L-\Lambda_q)f\oplus \gamma_{0,K}f,
\end{equation}
is Fredholm.
\end{prop}

\begin{proof}
We note that the strongly elliptic differential operator $L-\Lambda_q$ defines 
an elliptic boundary value problem when equipped with Dirichlet condition as 
in~\eqref{hetttillhminett}. 
It follows that the linear operator given in~\eqref{hetttillhminett} is Fredholm 
because it is a compact perturbation of the operator 
$f\mapsto (L+\lambda)f\oplus \gamma_{0,K}f$ which is invertible for $\Re(\lambda)$ 
large enough. 
\end{proof}

\begin{remark}
\label{splittingremark}
As a consequence of Proposition~\ref{fredholmofl}, there exists a bounded linear 
extension mapping $\mathfrak{Q}_q:H^{1/2}(\partial\Omega)\to H^1(K)$ 
satisfying 
$\gamma_0\mathfrak{Q}_q\phi=\phi$ and $(L-\Lambda_q)\mathfrak{Q}_q \phi=0$ in 
$K^\circ$ for $\phi$ outside a finite-dimensional subspace of 
$H^{1/2}(\partial\Omega)$.  This finite-dimensional subspace can be chosen as 
the orthogonal complement of the space 
$\gamma_{0,K}\left(\ker\left((L-\Lambda_q):H^1(K)\to H^{-1}(K)\right)\right)$. 
We conclude that there exists a finite rank operator $F_K$ on $H^1(K)$, 
smoothing 
in the interior, such that whenever $f\in H^1(K)$ satisfies $(L-\Lambda_q)f=0$ 
in $K^\circ$,
\[
\|f\|_{H^1(K)}\leq 
C\left(\|\gamma_0 f\|_{H^{1/2}(\partial\Omega)}+\|F_K f\|_{H^1(K)}\right),
\]
for some constant $C$. 
\end{remark}

\begin{lemma}\label{lem:Tn2}
Let $q\in\Nzero$ and let $F_q:L^2(\R^{2d})\to L^2(\partial\Omega)$ be the 
finite rank operator of Lemma~\ref{lem:Tn}, 
$S_{\tau,q}\in \Psi^{-\infty}(\partial\Omega)$ the finite rank operator of 
Lemma~\ref{controllingnonpositive} 
and $P_q$ the orthogonal projection on the Landau level 
$\mathcal L_q$. There exist constants $b>0$, $C>1$, 
such that for all 
$f\in L^2(\R^{2d})$ it holds that,
\begin{align*}
\frac1C\|\gamma_0&P_q f\|_{H^{1/2}(\partial\Omega)}
(\|\gamma_0P_qf\|_{H^{1/2}(\partial\Omega)}-
\|F_qf\|_{L^2(\partial\Omega)}
-b\|S_{\tau,q}\gamma_0P_qf\|_{L^2(\partial\Omega)})\\
&\leq  \langle f,T_qf\rangle_{L^2(\R^{2d})}
\leq C\|\gamma_0P_qf\|_{H^{1/2}(\partial\Omega)}
(\|\gamma_0P_qf\|_{H^{1/2}(\partial\Omega)}+\|F_qf\|_{L^2(\partial\Omega)})\,.
\end{align*}
\end{lemma}

\begin{proof}
Lemma~\ref{lem:pseudo-zero} states that the operator $T_{\tau,q}$ from 
Lemma~\ref{lem:Tn} is an elliptic pseudo-differential operator of order $1$ 
with positive principal symbol. 
Hence, by Lemma~\ref{controllingnonpositive}, there are constants $b,C>1$ such 
that
\[
\frac1C\|\phi\|_{H^{1/2}(\partial\Omega)}(\|\phi\|_{H^{1/2}(\partial\Omega)}
-b\|S_{\tau,q}\phi\|_{L^2(\partial\Omega)})
\leq \Real\langle \phi,T_{\tau,q} \phi\rangle_{L^2(\partial\Omega)}
\leq C\|\phi\|_{H^{1/2}(\partial\Omega)}^2,
\]
for all $\phi\in H^{1/2}(\partial\Omega)$. Applying the above estimates with 
$\phi=\gamma_0P_qf$
and $f\in L^2(\R^{2d})$, and recalling~\eqref{eq-lem:Tn}, we get
that the double inequality announced in
the above lemma holds for all $f\in L^2(\R^{2d})$ due to Lemma~\ref{lem:Tn} and 
the fact that 
\begin{align*}
\Real\langle \phi, T_{\tau,q}\phi\rangle_{L^2(\partial\Omega)}
&=\langle f,T_qf\rangle+\Real
\langle \phi,F_qf\rangle_{L^2(\partial\Omega)}. 
\end{align*}
\end{proof}

\begin{proof}[Proof of Lemma~\ref{lem:toeplitztva}]
Let $S_{\tau,q}\in \Psi^{-\infty}(\partial\Omega)$ denote the finite rank 
operator from Lemma~\ref{lem:Tn}.
The operator 
$S_{\tau,q}\gamma_0 P_q:L^2(\R^{2d})\to L^2(\partial\Omega)$ 
is a well defined finite rank operator since $S_{\tau,q}$ is finite rank and 
$P_qL^2(\R^{2d})\subseteq C^\infty(\R^{2d})$. Recall the finite rank operator 
$F_K$ on $H^1(K)$ from 
Remark~\ref{splittingremark} and let $F_KP_q:L^2(\R^{2d})\to H^1(K)$ denote the
finite rank operator 
$f\mapsto (P_qf)|_K\mapsto F_K\left[(P_qf)|_K\right]$ which is well defined 
since $(P_qf)|_K\in C^\infty(\overline{K})$.
We define the space 
$\mathcal{W}_q\subseteq \mathcal{L}_q$ by means of
\[
\mathcal{L}_q\ominus \mathcal{W}_q
:=\ker \left(S_{\tau,q}\gamma_0 P_q\right) \cap \ker F_q\cap \ker F_KP_q
\cap \mathcal{L}_q\subseteq \mathcal{L}_q.
\]
The space $\mathcal{W}_q$ is of finite dimension because all operators
$S_{\tau,q}\gamma_0 P_q$, $F_q$ and  $F_KP_q$ are of finite rank.

{\it Step~1. Lower bound.} 

We prove that  the lower bound in~\eqref{eq:inequality} is valid for all 
$f\in \mathcal{W}_q^\perp$. For simplicity we set $\phi:=\gamma_0 f$. 
By the definition of $T_q$ from~\eqref{eq-Tn}, the estimate of 
Lemma~\ref{lem:Tn2} gives,
\begin{displaymath}
\langle f,T_qf\rangle_{L^2(\R^{2d})}\geq
\frac1C\|\phi\|_{H^{1/2}(\partial\Omega)}^2\,.
\end{displaymath}
So it suffices to prove that
\begin{displaymath}
\langle f,S_q^{K}f\rangle_{L^2(\R^{2d})}\leq
C'\|\phi\|_{H^{1/2}(\partial\Omega)}^2\,,
\end{displaymath}
for some positive constant $C'$. Recalling the definition of $S_q^{K}$, and 
using that $\|f\|_{L^2(K)}\leq \|f\|_{H^{1}(K)}$, 
this follows once showing the estimate
\begin{equation}
\label{eq-lem-ppp}
\|f\|_{H^1(K)}\leq C'\|\phi\|_{H^{1/2}(\partial\Omega)}\,.
\end{equation}
Since $(L-\Lambda_q)f=0$ this estimate follows from Remark~\ref{splittingremark}.

\noindent{\it Step~2. Upper bound.}\\
Now we establish the upper bound in~\eqref{eq:inequality}. Let 
$f\in L^2(\R^{2d})$ and $u=P_qf$, the projection of $f$ onto the eigenspace 
$\mathcal L_q$. Notice that the trace 
theorem,~\cite[Theorem $9.4$, Chapter $1$]{lions} gives,
\begin{equation*}
\|\gamma_0 u\|_{H^{1/2}(\partial\Omega)}\leq
 C\|u\|_{H^2(K)}\,,
\end{equation*}
for some positive constant $C$. Notice that $(L-\Lambda_q)u=0$. By the elliptic 
regularity, given a domain
$K_1$ such that $K\subset K_1$, there exists a constant $C_{K_1}$ such that,
\begin{displaymath}
\|u\|_{H^2(K)}
\leq C_{K_1}
\left(\|L_qu\|_{L^2(K_1)}+\|u\|_{L^2(K_1)}\right)=C_{K_1}\|u\|_{L^2(K_1)}.
\end{displaymath}
Summing up, we get,
\begin{displaymath}
\|\gamma_0 P_qf\|_{H^{1/2}(\partial\Omega)}\leq
C\|P_qf\|_{L^2(K_1)}\,,\quad\forall~f\in L^2(\R^{2d})\,.
\end{displaymath}
Substituting the above inequality in the estimate of Lemma~\ref{lem:Tn2}, we 
obtain the upper bound announced in~\eqref{eq:inequality}.
\end{proof}

\appendix

\section{Proof of Lemma~\ref{lem:kernel}}
\label{sec:expansion}

The proof of Lemma~\ref{lem:kernel}, and the expansion of the function $G_0$ 
defined in~\eqref{gzeroexpression}, is based on an expansion of the function 
$I$.

\begin{lemma}
\label{iasymptotic}
The function 
\[
I(s):=\int_0^{+\infty} \frac{e^{-s\coth(t)}}{\sinh^d(t)}\, dt
\]
can be written as
\[
I(s)=I_0(s)+I_{\infty}(s)
\]
where 
\begin{enumerate}
\item $I_0(s),I_\infty(s)=O(e^{-s})$ as $s\to +\infty$.
\item $I_\infty\in C^\infty(\R)$.
\item The function $I_0\in C^\infty(\R_+)$ admits an asymptotic expansion for 
small $s$:
\begin{equation}
\label{izeronasty}
I_0(s)=\sum_{j=1-d}^{{+\infty}} (e^{-sa}c_j-c_j')s^j
-\sum_{j=0}^{+\infty} d_j s^j\log(s),
\end{equation}
where 
\begin{align*}
c_j&:=\begin{dcases}
\smashoperator[r]{\sum_{k=0}^{\lfloor (d-j)/2\rfloor-1}} (-1)^k
\dbinom{\frac{d-2}{2}}{k}
(d-2(k+1))_{j+1}a^{d-2k-1-j},& 1-d\leq j<0,\\
\smashoperator[r]{\sum_{k=\lceil (d-j-1)/2\rceil-1}^{+\infty}} (-1)^k
\dbinom{\frac{d-2}{2}}{k}
\frac{a^{d-2k-1+j}}{(2k+1-d)_{j+1}},& 0\leq j\leq d-2,\\
0,& j>d-2.
\end{dcases}\\
c_j'&:=
\begin{dcases}
\smashoperator[r]{\sum_{\substack{j\geq k\geq 1\\j-k-d \equiv 1 \bmod 2}}}
\frac{(-1)^{(j+k+d-1)/2}a^k}{(j-k)!k\cdot k!} 
\dbinom{\frac{d-2}{2}}{\frac{j-k+d-1}{2}}\\
\qquad +\frac{(-1)^{(j+d-1)/2}}{j!} 
\dbinom{\frac{d-2}{2}}{\frac{j+d-1}{2}}
(\gamma+\log(a)+1),& j-d\equiv 1 \bmod 2, \;j\geq 1,\\
\smashoperator[r]{\sum_{\substack{j\geq k\geq 1 \\ j-k-d\equiv 1\bmod 2}}}
\frac{(-1)^{(j+k+d-1)/2}a^k}{(j-k)!k\cdot k!} 
\dbinom{\frac{d-2}{2}}{\frac{j-k+d-1}{2}}, 
& j-d\equiv 0 \bmod 2, \;j\geq 1.\\
\gamma+\log(a)+1,& \text{$j=0$ and $d$ is odd.}\\
0,& \text{$j=0$ and $d$ is even.}
\end{dcases}\\
d_j&:=\smashoperator{\sum_{\substack{j\geq k\geq 1\\j-k-d\equiv 1\bmod 2}}} 
\frac{(-1)^{(j+k+d-1)/2}a^k}{(j-k)!k\cdot k!} 
\dbinom{\frac{d-2}{2}}{\frac{j-k+d-1}{2}}.
\end{align*}
Here $a=\coth(1)$ and $\gamma$ is the Euler-Mascheroni constant.
\end{enumerate}
\end{lemma}

\begin{proof}
We write
\[
I_0(s):=\int_0^1 \frac{e^{-s\coth(t)}}{\sinh^d(t)} \,dt
\quad\mbox{and}\quad I_\infty(s):=\int_1^{+\infty}  
\frac{e^{-s\coth(t)}}{\sinh^d(t)}\,d t.
\]
It is easily verified that $I_\infty\in C^\infty(\R)$. Since 
$\coth(t)\geq 1$ for $t\geq 1$, $I_\infty(s)=O(e^{-s})$ as 
$s\to +\infty$. Since $\coth(t)\geq 6/5$ for $t\in[0, 1]$, 
$I_0(s)=O(e^{-s})$ as 
$s\to +\infty$ follows once one notice that the singularity at $t=0$ from the 
$\sinh^d(t)$ is canceled by $e^{-s\coth(t)/6}$. 

Let us turn to the unpleasant computation of $I_0$. After the change 
of variables $u=\coth(t)$ the integral defining $I_0$ transforms to
\[
I_0(s)=\int_a^{+\infty} e^{-su}(u^2-1)^{\frac{d-2}{2}}\, d u
=\int_a^{+\infty} e^{-su}u^{d-2}
\Bigl(1-\frac{1}{u^{2}}\Bigr)^{\frac{d-2}{2}}\, d u.
\]
Here $a:=\coth(1)>1$. Using the Taylor expansion of 
$\left(1-\frac{1}{u^{2}}\right)^{\frac{d-2}{2}}$, which converges uniformly 
for $u\geq \coth(1)$, we arrive at the identity
\begin{align}
\label{izeroaftertaylor}
I_0(s)&=\sum_{k=0}^{+\infty} (-1)^k
\dbinom{\frac{d-2}{2}}{k}
\int_a^{+\infty} e^{-su} u^{d-2(k+1)}\, d u\\
\nonumber
&=\sum_{k=0}^{+\infty} (-1)^k
\dbinom{\frac{d-2}{2}}{k}
s^{2k+1-d}g_{d-2(k+1)}(sa).
\end{align}
where
\[
g_m(t):=\int_{t}^{+\infty} e^{-u} u^{m}\, d u.
\]
After an integration by parts, one arrives at the identity
\[
g_m(t)=
\begin{dcases}
e^{-t}\sum_{j=0}^m(m)_{j+1} t^{m-j},& m\geq0,\\
e^{-t}\sum_{j=1}^{-m-1}\frac{t^{m+j}}{(-m-1)_{j}}
-\frac{\gamma+\log(t)
+\sum_{j=0}^{+\infty} (-1)^j\frac{t^j}{j\cdot j!}}{(-m-1)!},& m<0.
\end{dcases}
\]
Here $(m)_l:=m(m-1)\cdots (m-l+1)$ denotes the Pochhammer symbol. Putting this 
into~\eqref{izeroaftertaylor}, 
\begin{align*}
I_0(s)=&\sum_{k=0}^{\lfloor (d-2)/2\rfloor}
\sum_{j=0}^{d-2(k+1)} (-1)^k 
\dbinom{\frac{d-2}{2}}{k}
(d-2(k+1))_j e^{-sa} a^{d-2(k+1)-j}s^{-j-1}\\
&+\sum_{k=\lfloor d/2\rfloor}^{+\infty} (-1)^k
\dbinom{\frac{d-2}{2}}{k}
e^{-sa}
\sum_{j=1}^{2(k+1)-d} \frac{a^{d-2(k+1)-j}}{(2k+1-d)_j}s^{j-1}\\
&-\sum_{k=\lfloor d/2\rfloor}^{+\infty} (-1)^k
\dbinom{\frac{d-2}{2}}{k}
\frac{s^{2k+1-d}(\gamma+\log(a)+1)+s^{2k+1-d}\log(s)}{(2k+1-d)!}\\
&-\sum_{k=\lfloor d/2\rfloor}^{+\infty} 
\sum_{j=1}^{+\infty}(-1)^{k+j}
\dbinom{\frac{d-2}{2}}{k}
\frac{a^j}{j\cdot j!}\frac{s^{2k+1-d+j}}{(2k+1-d)!}.
\end{align*}
Rearranging these terms leads to the expression~\eqref{izeronasty}.
\end{proof}

\begin{proof}[Proof of Lemma~\ref{lem:kernel}]
Using Lemma~\ref{iasymptotic}, we have that 
\begin{align*}
G_0(z,\zeta)&=\sum_{j=1-d}^{{+\infty}} \frac{b^{j+d-1}}{2^{2d+2j-1}\pi^d} 
e^{b(\overline{z}\cdot \zeta-\overline{\zeta}\cdot z)/4}
\bigl(e^{-ab|z-\zeta|^2/4}c_j-c_j'\bigr)|z-\zeta|^{2j}\\
&\qquad-\sum_{j=0}^{+\infty} d_j  \frac{b^{j+d-1}}{2^{2d+2j-1}\pi^d} 
e^{b(\overline{z}\cdot \zeta-\overline{\zeta}\cdot z)/4}|z-\zeta|^{2j}
\log\Bigl(\frac{b|z-\zeta|^2}{4}\Bigr)\\
&\qquad + \frac{2b^{d-1}}{(4\pi)^d} 
e^{b(\overline{z}\cdot \zeta-\overline{\zeta}\cdot z)/4}
I_\infty\Bigl(\frac{b|z-\zeta|^2}{4}\Bigr).
\end{align*}
From these expressions, the Lemma follows.
\end{proof}

\section*{Acknowledgements}
The authors wish to thank Grigori Rozenblum who suggested the approach we used 
from the theory of pseudo-differential operators.
Part of this work has been prepared in the Erwin Schr\"odinger 
Institute (ESI) - Vienna which is gratefully acknowledged. The authors MG and 
MPS also wish to thank the 
Mittag--Leffler institute for a productive stay in 2012. AK is supported by 
a grant from Lebanese University.

\end{document}